\documentclass[12pt]{article}
\usepackage[left=2.5cm,right=2cm,     top=2.5cm,bottom=2cm,bindingoffset=0cm]{geometry}
\usepackage{amsthm}
\usepackage{graphicx,textcomp}
\usepackage{amsmath,bm,amsfonts,mathrsfs,amssymb,textcomp}
\usepackage{hyperref}
\newtheorem{theorem}{Theorem}

\newtheorem{prop}[theorem]{Proposition}
\newtheorem{cor}[theorem]{Corollary}
\usepackage{graphicx,color}
\usepackage{enumitem}   
\usepackage{xcolor}

\usepackage[normalem]{ulem}

\begin{document}
\title{On non-homeomorphic surfaces with close DN maps}
\author{D.V. Korikov\thanks{PDMI RAS} \thanks{ITMO University} \thanks{e-mail: thecakeisalie@list.ru, \ ORCID:\href{https://orcid.org/0000-0002-3212-5874}{0000-0002-3212-5874}} \thanks{This work was supported by the Ministry of Science and Higher Education of the Russian Federation (agreement 075-15-2025-344 dated 29/04/2025 for Saint Petersburg Leonhard Euler International Mathematical Institute at PDMI RAS).}}
\maketitle
\begin{abstract}
Let $(M,g)$ be a genus $m$ surface with boundary $\Gamma$ and DN map $\Lambda$. Introduce the Schottky double $2M$ of $(M,g)$ and denote by $Sys(2M)$ the length of the shortest closed geodesics in the hyperbolic metrics on $2M$. We prove that $Sys(2M)$ is small if $\Lambda$ is close, in the operator norm, to the DN map $\Lambda_*$ of some surface $(M_*,g_*)$ of lower genus $m_*<m$ with the same boundary $\Gamma$:
$$\|\Lambda-\Lambda_*\|_{B(H^{1/2}(\Gamma);H^{-1/2}(\Gamma))}\to 0\,\Longrightarrow \ Sys(2M)\to 0.$$
\end{abstract}




\paragraph{Stability of determination of a surface from its DN map.} Let $(M,g)$ be an orientable surface (smooth two-dimensional Riemannian manifold) of genus $m$ with boundary $(\Gamma;dl)$ (in what follows, we assume that $\Gamma$ is diffeomorphic to a circle while the length element induced by the metric $g$ on $\Gamma$ coincides with $dl$). Introduce the Laplace-Beltrami operator $\Delta_g$ and denote by $u^f$ the harmonic extension of $f\in C^\infty(\Gamma)$ into $(M,g)$. Let $\nu$ be a unit outward normal on $\Gamma$. The {\it Dirichlet-to-Neumann map} (DN map) $\Lambda$ of $(M,g)$ is given by $\Lambda f:=\partial_\nu u^f|_\Gamma$. 

\smallskip

Let $(M',g')$ be another surface with boundary $(\Gamma',dl')$, Laplace-Beltrami operator $\Delta_{g'}$, and DN-map $\Lambda':\,f'\mapsto \partial_{\nu'}u^{'f'}|_{\Gamma'}$, where $u^{'f'}$ is the harmonic extension of $f'\in C^\infty(\Gamma';\mathbb{R})$ into $(M',g')$ and $\nu'$ is the exterior normal. If there is a confromal diffeomorphism $\beta:\,M\to M'$, $\beta^*g'=\rho g$, then the harmonic functions and DN-maps of $(M,g)$ and $(M',g')$ are connected via 
\begin{equation}
\label{DN conformal gauge}
\beta^*[u^{'f'}]=u^{\beta^*f'}, \qquad \Lambda'=\beta^{*-1}\frac{1}{\sqrt{\rho}}\Lambda\beta^*,
\end{equation}
where $\beta^*u':=u'\circ\beta$ is a precomposition. In particular, $\Lambda'=\Lambda$ if $(\Gamma',dl')=(\Gamma,dl)$ and $\beta$ does not move points of $\Gamma$. The well-known result of Lassas and Uhlmann \cite{LU} states that the converse is also true, i.e., the equalities $\Gamma=\Gamma'$ and $\Lambda'=\Lambda$ imply the existence of conformal diffeomorphism $\beta$ between $(M,g)$ and $(M',g')$ that does not move the points of $\Gamma$.

\smallskip

In \cite{BKor JIIPP stability,BKTeich}, it is proved that the closeness of the DN map $\Lambda$ of a surface $(M',g')$ (with the same genus and boundary as $(M,g)$) to the DN map $\Lambda$ in the operator norm implies the closeness of the conformal classes $\tau':=[(M',g')]$ and $\tau:=[(M,g)]$ of $(M',g')$ and $(M,g)$ in the Teichm\"uller metric $d_T$; the estimates of 
$$d_T(\tau',\tau)\le c(\Lambda)\|\Lambda'-\Lambda\|_{B(H^1{\Gamma},L_2(\Gamma)}$$ 
were derived in \cite{Kor CAOT}. This means that, if $\Lambda'$ is sufficiently close to $\Lambda$ and $M'$ is homeomorphic to $M$, then there is a diffeomorphism $\beta:\,M\mapsto M'$ and the positive function $\rho'$ on $M'$ such that $\beta$ is the identity on $\Gamma=\Gamma'$ and $$\|\beta^*(\rho'g')-g\|_{C(M)}\le c(\Lambda)\|\Lambda'-\Lambda\|_{B(H^1{\Gamma},L_2(\Gamma;dl)}.$$

\paragraph{Instability of surface topology under small perturbations of its DN map.} As noted above, the conformal class of a surface depends continuously on its DN map provided that the surface topology is fixed. In contrast to this, the topology of a surface is unstable under small perturbations of its DN map \cite{Kor ZNS}. Namely, by cutting small holes in the surface and then attaching small handles or M\"obius strips to the hole boundaries, one can obtain a surface whose DN map is arbitrarily close (in the operator norm) to the DN map of the original surface. Note that one cannot lower the genus of the surface or make orientable surface from a non-orientable one without significant perturbation of its DN map. For some other results on the stability/instability of the surface topology under a small perturbation of the DN map we refer the reader to Propositions 5.5, 6.2 and Theorem 6.3, \cite{Kor CAOT}.

\smallskip

This note addresses the question of whether the converse is true, namely, whether the closeness of the DN map $\Lambda'$ of a topologically perturbed surface $(M,g)$ to the DN map $\Lambda_*$ of $(M_*,g_*)$ implies that all the ``extra'' handles on $(M,g)$ are effectively separated by small closed curves from the ``exterior'' part of the surface. However, the DN map does not change if one enlarges small handles by multiplying the surface metric $g$ by a conformal factor which is large on the ``extra'' handles and equal to one on the boundary. Thus, to make the above discussion correct, the length of closed curves on $M$ should be defined in some ``canonical'' metric $h$ belonging to the same conformal class $[g]$ as the original metric $g$ on $M$. 

\begin{figure}[h!]
\center{\includegraphics[width=0.7\linewidth]{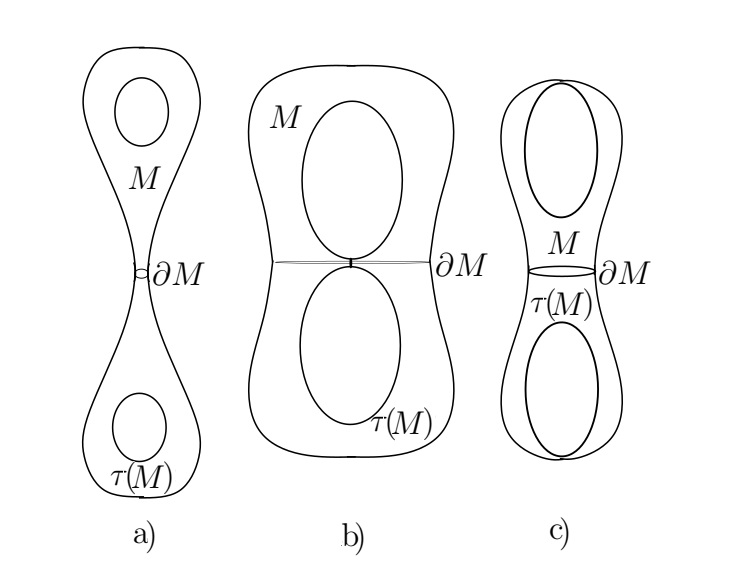}}
\caption{Different types of degeneration of $(2M,h)$ as $\Lambda\to\Lambda_*$, the genus 2 case.}
\label{degpic}
\end{figure}

\smallskip

The most natural choice is the {\it hyperbolic metric} on $M$ which is the metric $h\in[g]$ of constant scalar curvature $K=-1$ on $M$ such that the boundary $\partial M=\Gamma$ is a geodesic curve on $(M,h)$. This metric can be constructed as follows. As a rule, we assume that $M$ is endowed with the complex structure (maximal complex atlas) associated with the conformal class $[g]$ and the choice of the orientation on $M$. By attaching to $M$ its copy endowed with the opposite orientation, one obtains the {\it Schottky double} $2M$ of $M$ which is the Riemann surface endowed with the antiholomorphic involution $\tau$ that interchanges the points of the original and copy in such a way that $2M/\tau\cong M$ and $\partial M$ coincides with the set of fixed points of $\tau$. There is the unique (hyperbolic) metric $h$ on $2M$ of constant scalar curvature $-1$ compatible with the complex structure. Thus, $\tau^{'*}h=h$, i.e. the involution $\tau$ is an isometry on $(2M,h)$. Due to the last symmetry, a geodesic curve on $(2M,h)$ starting from any point $\partial M\subset 2M$ with the same tangent vector as $\partial M$, does not leave $\partial M$. So, the $\partial M$ is geodesic and the restriction of $h$ to $M$ (still denoted by $h$) is the required hyperbolic metric on $M\subset 2M$.

\smallskip

(Note that the embedding $(\Gamma,dl)$ into $2M$ is $C^{1+r}$-smooth if $g$ is $C^r$-smooth ($r\in (1,2)$). Indeed, near $\Gamma$ one can choose the local holomrphic coordinates $z(x):=iG(x,y)-\int_\cdot^x\star d_x G(x,y)$, where $y\in M\backslash\Gamma$, $G$ is the Green function of the Dirichlet Laplacian, and $\star$ is the Hodge star operator on $(M,g)$. Since $g$ is $C^r$-smooth, $G(\cdot,y)$ is $C^{1+r}$-smooth outside $y$ due to the local Schauder estimates of solutions to elliptic boundary value problems; see Theorem 6.2 and Lemma 6.4, \cite{Trudinger}. Then so is $z$, while the extension of $z$ onto the double $2M$ by the rule $z(\tau(x))=\overline{z(x)}$ provides the holomorphic coordinates on $2M$ that are $C^{1+r}$-smooth on each $M$, $\tau(M)$. Similarly, the embedding $(\Gamma,dl)$ into $2M$ is $C^{\infty}$-smooth if $g$ is $C^\infty$-smooth.) 

\smallskip

The {\it length spectrum} ${\rm Sp}(2M)$ of $2M$ is a collection of lengths of all closed geodesics in hyperbolic metric on $2M$. Denote by $Sys(2M)$ the infimum of ${\rm Sp}(2M)$. The result of this note is the following statement.

\begin{prop}
\label{main prop}
Let $\{(M_k,g_k)\}_{k=1}^{\infty}$ be a sequence of genus $m$ surfaces of  with the same {\rm(}connected{\rm)} boundary $(\Gamma,dl)$and let $\Lambda_k$ be DN map of $(M_k,g_k)$ and $2M_k$ be a double of $M_k$ equipped with the hyperbolic metric $h_k$. Suppose that 
\begin{equation}
\label{refined closeness}
\|\Lambda_k-\Lambda_*\|_{B(H^{1/2}(\Gamma);H^{-1/2}(\Gamma))}\to 0,
\end{equation}
where $\Lambda_*$ is a DN map of some surface $(M_*,g_*)$ of genus $m_*<m$ with boundary $(\Gamma,dl)$. 
Then $$Sys(2M_k)\to 0.$$
\end{prop}
This result can be interpreted as follows. Suppose that $\Lambda$ is close to the DN map of a surface of lower genus. Then there is the closed geodesic $\mathscr{L}$ on the double $(2M,h)$ of $M$ of small length $L$; without loss of generality, one can assume that $\mathscr{L}$ is simple. Due to the Collar Lemma (see, e.g., formula (A) in \cite{Buser}, Lemma 13.6, \cite{Farb}, and the proof of Corollary 4.4, \cite{Wang}) the ${\bf d}$-neighborhood $\mathscr{R}(\mathscr{L})$ of $\mathscr{L}$ is biholomorphic to a thin cylinder $(0,1)\times\mathbb{T}_{r}$, where $\mathbb{T}_r=\mathbb{R}/2\pi r$ and
$${\bf d}={\rm arcsinh}\Big(1\Big/{\rm sinh}\Big(\frac{L}{2}\Big)\Big)\simeq |{\rm log}(L)|\gg 1, \qquad r=\frac{L}{2\pi^2}\ll 1.$$
Since the current flowing across the curve is a conformal invariant, the intersection of $\mathscr{R}(\mathscr{L})$ with $M\subset 2M$ can be considered as a ``thin'' domain which suppresses the currents flowing through it, thereby effectively reducing the surface genus observed from the boundary by electric impedance tomography (see Fig. \ref{degpic}).

\paragraph{Connection between the spectrum of the Hilbert transform of a surface and the length spectrum of its double.} Let $(M,g)$ be a smooth surface with the connected boundary $(\Gamma,dl)$ and the DN map $\Lambda$. Choose a unit tangent vector $\gamma$ on $(\Gamma,dl)$; in what follows, we assume that the orientation on $(M,g)$ with boundary $(\Gamma,dl)$ is chosen in such a way that the pair $(\nu,\gamma)$ is positively oriented, where $\nu$ is the exterior normal. Denote by $\partial_\gamma$ the differentiation along the length on $\Gamma$ in the direction $\gamma$ and introduce the operator $\partial_\gamma^{-1}$ that vanishes on constants and inverts $\partial_\gamma$ on $L_2(\Gamma;dl)\ominus\mathbb{C}=\partial_\gamma H^{1}(\Gamma)$. The operator 
$$H:=\partial_\gamma^{-1}\Lambda$$ 
is called the {\it Hilbert transform} of $(M,g)$. As easily seen, $H$ admits an extension to a continuous operator acting in $H^{1/2}(\Gamma)$. There is the following connection between $H$ and the topology and the complex structure on $(M,g)$. Restricting the Cauchy-Riemann equations $d\Im\,w=\star d\Re\,w$ in $M$ to the boundary $\Gamma$, on obtains $\partial_\gamma\Im\,\eta=\Lambda\Re\,\eta$, $\partial_\gamma\Re\,\eta=-\Lambda\Im\,\eta$, where $\eta=w|_\Gamma$. Integration of these equations yield $\Im\,\eta=H\Re\,\eta+{\rm const}$, $(H^2+I)\Re\eta={\rm const}$. Thus, the Hilbert transform is the operator relating the real and imaginary parts of boundary traces of holomorphic functions on the surface. The eigenfunctions of $H$ corresponding to the eigenvalues $-i$ ($+i$) are all the boundary traces $\eta$ of holomorphic (antiholomprhic) functions $w$ on $(M,g)$ obeying $\eta\in H^{1/2}(\Gamma)$ and $(\eta,1)=0$. Outside $\lambda=\pm i$, the spectrum of $H$ is discrete and consists from zero (${\rm Ker}H=\mathbb{C}$) and the (semi-simple) eigenvalues 
$$\lambda_{\pm k}(H)=\pm\mu_k i(H), \qquad 0<\mu_1(H)\le\dots\le\mu_m(H)<1,$$
of total multiplicity $2m$, where $m$ is the genus of $M$ (Lemma 1, \cite{B}, see also Lemma 2.1, \cite{KorCJM}). The corresponding eigenfunctions $\eta_{\pm k}$ are $C^\infty$-smooth if so is the metric $g$.

\smallskip

The definition of the Hilbert transform implies its conformal invariance: if $\beta:\,(M,g)\to (M',g')$ is a conformal orientation-preserving (not necessarily smooth) diffeomorphism, $\beta^*g'=\rho g$, then from (\ref{DN conformal gauge}) and the rule $\partial_{\gamma'}=\beta^{*-1}\frac{1}{\sqrt{\rho}}\partial_{\gamma}\beta^*$ for differentiations along $(\Gamma'=\partial M',dl'=g'|_{T^*\Gamma'\otimes T^*\Gamma'})$ and $(\Gamma,dl)$, it follows that the Hilbert transforms $H$ and $H'$ of $(M,g)$ and $(M',g')$ are related via
\begin{equation}
\label{Hilbert conformal gauge}
\begin{split}
H'f&=\beta^{*-1}H\beta^*f+{\rm const}(f).
\end{split}
\end{equation}
Here the additional constants appear since $\beta^{*-1}:\,L_2(\Gamma,dl)\mapsto L_2(\Gamma',dl')$ is not isometric and, thus, it does not map $L_2(\Gamma,dl)\ominus\mathbb{C}$ (the image of $\partial_\gamma^{-1}$) to $L_2(\Gamma',dl')\ominus\mathbb{C}$ (the image of $\partial_{\gamma'}^{-1}$). In particular, the spectrum of is conformal invariant: indeed, $H\eta_k=\lambda_k(H)\eta_k$ ($\lambda_k\ne 0$) implies $H'(\beta^{*-1}\eta_k+{\rm const}(\eta_k)/\lambda_k(H))=\lambda(\beta^{*-1}\eta_k+{\rm const}(\eta_k)/\lambda_k(H))$.

In particular, the above spectral properties of $H$ remain valid even if the metric $g$ on $M$ is only $C^r$-smooth ($r\in(1,2)$). Indeed, as shown above, in this case there is the $C^{r+1}$-smooth embedding $\beta$ of $(M,g)$ into the its double $2M$ whence connection (\ref{Hilbert conformal gauge}) is valid for any Hilbert transform $H'$ of the surface $\beta(M)$ with the (smooth) boundary $\partial\beta(M)$.

\smallskip

Since the total multiplicity of the eigenvalues $\lambda_k(H)$ of $H$ different from $\pm i,0$ coincide with the doubled genus $m$ of a surface $(M,g)$, the pair $\lambda_{\pm m}$ of them tends to $\pm i$, respectively, as the surface degenerates (for example, one of its handles becomes small, as in the example of \cite{Kor ZNS}). In the rest of this note, we prove that the converse is also true in the following sense.
\begin{prop}
\label{real prop}
Let $\{(M_k,g_k)\}_{k=1}^{\infty}$ be a sequence of genus $m$ surfaces of  with the same {\rm(}connected{\rm)} boundary $(\Gamma,dl)$ and let $H_k$ be Hilbert transforms of $(M_k,g_k)$ and $2M_k$ be a double of $M_k$ equipped with the hyperbolic metric $h_k$. Denote by $\lambda_{m}(H_k)\ne i$ the eigenvalue of $H_k$ closest to $i$. Then 
$$\lambda_{m}(H_k)\to i \ \Rightarrow \ Sys(2M_k)\to 0.$$
\end{prop}
Note that, in the case $m_0=0$, $m=1$, Proposition \ref{real prop} is proved in \cite{arxiv small handles gen 1} by the use of explicit expressions for $b$-period matrices of the doubles $2M_k$ in terms of the Hilbert transforms $H_k$ provided by the results of \cite{KorCJM}. In addition, Proposition \ref{main prop} is a corollary of Proposition \ref{real prop}. Indeed, (\ref{refined closeness}) and the definition of Hilbert transforms imply
\begin{equation}
\label{hilbert closeness}
\|H_k-H_*\|_{B(H^{1/2}(\Gamma))}\to 0,
\end{equation}
where $H_*=\partial_\gamma^{-1}\Lambda_*$ is the Hilbert transform of $(M_*,g_*)$. Since the discrete spectra of the operators in Banach space are stable under small (in the operator norm) perturbations (see IV.3.5, Kato \cite{Kato}), formula (\ref{hilbert closeness}) implies that, in any domain $\mathscr{D}$ whose closure does not intersect the essential spectrum $\{-i,+i\}$ of $H_*$, the total multiplicities of eigenvalues of $H_k$ and $H_*$ lying in $\mathscr{D}$ do coincide for sufficiently large $k>k(\mathscr{D})$. Thus, the total multiplicity of all $\lambda_k(M_k)\in\mathbb{D}\backslash\{0\}$ is equal to $2m_*$ while the remaining eigenvalues $\lambda_{\pm (m_*+1)}(M_k),\dots,\lambda_{\pm m}(M_k)$ tend to $\pm i$, respectively.

\begin{proof}[Proof of Proposition \ref{real prop}] We prove Proposition \ref{real prop} by contradiction. Suppose that $\lambda_m(H_k)\to i$ but $Sys(2M_k)\not\to 0$; then, passing to the sub-sequence, one can assume that $Sys(2M_k)\ge\epsilon$ for all $k$, where $\epsilon>0$. Denote by $\mathcal{M}_m(\epsilon)$ the space of genus $m$ Riemann surfaces $M$ with connected smooth boundaries whose Schottky doubles $2M$ obey $Sys(2M)\ge \epsilon$. Then $\{M_k\}_{k=1}^\infty\subset\mathcal{M}_m(\epsilon)$.

\smallskip

1) Recall that, according to the well-known Mumford compactness criterion (Corollary 3, \cite{Mum}), for each $\epsilon>0$, $g>1$, the set $\{[X]\in\mathcal{M}_g \ | \ {\rm Sys}(X)>\epsilon\}$ is compact in the moduli space $\mathcal{M}_g$ of genus $g$ Riemann surfaces without boundaries. In what follows, we use the following natural extension of this criterion for general manifolds with boundaries obtained in Theorem 3.1 \cite{AKKYLT}. For a $d$-dimensional smooth compact connected Riemannian manifold $(M,g)$ (possibly, with non-empty smooth boundary) and a point $x\in M$, let us introduce the (radial) geodesic coordinate $\alpha_x$ in the neighborhood of $x$ by $\alpha_x(y):=k$, where $k\in T_x M$ and $y={\rm exp}_x(k)$ is the end of the geodesic $\gamma_{x,\nu}:\,[0,1]\to M$ in $(M,g)$ with starting point $x$ and initial tangent vector $k$. The maximal $r\ge 0$ such that $\alpha_x$ provides an embedding of the neighborhood ${\rm dist}_g(y,x)<r$ in $(M,g)$ into $T_x M$ is called the {\it injectivity radius} of $(M,g)$ at $x$; its supremum over all $x\in M$ is denoted by $\iota_{(M,g)}$. If $\partial M\ne\varnothing$, then we introduce the boundary normal coordinates $\alpha_{\partial M}$ near $\partial M$ by $\alpha_{\partial M}(y):=n$, where $n\in N_x\partial M\subset N\partial M$ is the normal vector to $\partial M$ such that $y$ is the end of the geodesic $\gamma_{x,n}:\,[0,1]\to M$ in $(M,g)$ with starting point $x$ and initial tangent vector $n$. The {\it boundary injectivity radius} $\iota^{b}_{(M,g)}$ of $(M,g)$ is the maximal $r$ such that $\alpha_{\partial M}$ provides an embedding of the boundary collar ${\rm dist}_{g}(y,\partial M)<r$ into $N\partial M$. Then, for any $R_0,H_0,\iota_0,D_0>0$, the set $\mathcal{M}(d,R_0,H_0,\iota_0,D_0)$ of smooth $d$-dimensional manifolds $(M,g)$ obeying
\begin{align}
\label{curvatures}
\|{\rm Ric}_{(M,g)}\|_{L^{\infty}(M)},\,\|{\rm Ric}_{(\partial M,g|_{T\partial M})}\|_{L^{\infty}(\partial M)}\le R_0, \quad \|H_{\partial M}\|_{Lip(\partial M)}\le H_0,\\
\label{injectivities}
\iota_{(M,g)},\,\iota_{(\partial M,g|_{T\partial M})},\iota^{b}_{(M,g)}\ge\iota_0, \qquad {\rm diam}(M,g)\le D_0
\end{align}
(where ${\rm Ric}$ is the Ricci tensor of $(M,g)$, $|{\rm Ric}_{M}|:=\sqrt{R_{ij}R^{ij}}$ and $H$ is the mean curvature of $H_{\partial M}$) is pre-compact in $C^r$-topology ($0<r<2$): any sequence in $\mathcal{M}(n,r,R_0,H_0,\iota_0,D_0)$ contains a sub-sequence $\{(M_j,g_j)\}_{j=1}^{\infty}$ such that there are a $C^r$-smooth manifold $(M_\infty,g_\infty)$ and $C^r$-smooth diffeomorhpisms $\beta_j:\,M_\infty\to M_j$ obeying $\beta_j^*g_j\to g_\infty$ in $C^r(M_\infty,T^*M_\infty\otimes T^*M_\infty)$. Note that one can assume that the manifold $M_\infty$ itself is $C^\infty$-smooth (see, e.g., Theorem 5.11, \cite{Munk}) while $g_\infty\in\cap_{0<r<2}C^r(M_\infty,T^*M_\infty\otimes T^*M_\infty)$.

\smallskip

The above statement implies the following generalization of the Mumford comacteness criterion. Suppose that the metric on each surface $M\in\mathcal{M}_m(\epsilon)$ is obtained by restriction on $M$ of the hyperbolic metric $h$ on its Schottky double $2M$. Due to the connection ${\rm Ric}_{(M,g)}=K_g g$ between the Ricci and Gaussian curvatures valid in the two-dimensional case, one has ${\rm Ric}_{(M,h)}=-h$ and $\|{\rm Ric}_{M}\|_{L^{\infty}(M)}^2=2$. At the same time, in one dimension the Ricci curvature is zero $|{\rm Ric}_{(\partial M,h|_{T\partial M})}|=0$, and, since $\partial M$ is geodesic in the hyperbolic metric, its mean curvature is zero due to see Theorem 10.5.1, \cite{Bishop}. So, conditions (\ref{curvatures}) hold with $R_0=\sqrt{2}$ and $H_0=0$. Next, let $2M$ be a double of $M$ equipped with the hyperbolic metric $h$ and the anti-holomorphic (thus, isometric) involution $\tau$. Since there are no conjugate points on surfaces with negative curvature, the injectivity radius of $2M$ is a half of the length of the shortest closed geodesic curve on it, $\iota_{(2M,h)}=Sys(2M)/2$ while $\iota_{(M,h)}\ge \iota_{(2M,h)}\ge\epsilon/2$. Obviously, $\iota_{(\partial M,h|_{T\partial M})}\ge Sys(2M)/2\ge\epsilon/2$ since $\partial M$ is the close geodesic curve in $(2M,h)$. At the same time, if $L_1$ and $L_2$ are two geodesic segments starting from $x_1\in\partial M$ and $x_2\in\partial M$ in the normal direction, respectively, and ending at the common point $x$, then $L_1,L_2$ and $\tau(L_1),\tau(L_2)$ constitute the closed piece-wise geodesic curve $L$ whose length should be greater than $Sys(2M)$. This means that $\iota^{b}_{(M,h)}\ge Sys(2M)/4\ge\epsilon/4$. Finally, from the estimate
\begin{align*}
{\rm diam}(X,h)\,Sys(X)&\le C_0(g){\rm Vol}(X,h)=\\
=-C_0(g)\int_{X}K_h&\,dS_h=-2\pi\chi(X)C_0(g)=4\pi(g-1)C_0(g)=:C_1(g)
\end{align*}
valid for Riemann surfaces of genus $g>1$ endowed with the hyperbolic metrics $h$ (see lemma at p. 291, \cite{Mum}), one has ${\rm diam}(M,h)\le {\rm diam}(X,h)\le C_1(2m)/\epsilon$. Thus, conditions (\ref{injectivities}) hold with $\iota_0=\epsilon/4$ and $D_0=C_1(2m)/\epsilon$, whence the set 
$$\mathcal{M}_m(\epsilon)\subset\mathcal{M}(2,\sqrt{2},0,\epsilon/4,C_1(2m)/\epsilon)$$ 
is pre-compact in $C^r$-topology ($r\in (0,2)$). This means that each sequence in $\mathcal{M}_m(\epsilon)$ contains a sub-sequence $\{M_k\}_{k=1}^\infty$ such that there are a genus $m$ surface $(M_\infty,h_\infty)$ and $C^r$-diffeomorphisms $\beta_k:\,M_\infty\to M_k$ obeying $\beta_k^* h_k\to h_\infty$ in $C^r(M_\infty,T^*M_\infty\otimes T^*M_\infty)$, where $h_k$ is the hyperbolic metric on (the Schottky double of) $M_k$. So one can assume (passing to the sub-sequences) that this is true for the sequence $\{M_k\}_{k=1}^\infty\subset\mathcal{M}_m(\epsilon)$ from Proposition \ref{real prop}.

\smallskip

2) Denote $\Gamma_\infty:=\partial M_\infty$, $\tilde{h}_k:=\beta_k^* h_k$ and introduce the Laplacians 
$$\mathfrak{L}_k:=\frac{1}{\sqrt{|\tilde{h}_k|}}\partial_i\sqrt{|\tilde{h}_k|}\,\tilde{h}_k^{ij}\partial_j, \qquad \mathfrak{L_\infty}:=\frac{1}{\sqrt{|h_\infty|}}\partial_i\sqrt{|h_\infty|}\,h_\infty^{ij}\partial_j$$
on $(M_\infty,\beta_k^* h_k)$ and $(M_\infty,h_\infty)$, respectively; then $\lambda_-h_\infty^{ij}\xi_i\xi_j\le\tilde{h}_k^{ij}\xi_i\xi_j\le\lambda_+h_\infty^{ij}\xi_i\xi_j$ for any $\xi\in T^*M_\infty$, whence $\sqrt{|\tilde{h}_k|/|h_\infty|}\in[\lambda_+^{-1},\lambda_-^{-1}]$. Denote $Q_k:=\|\tilde{h}_k-h_\infty\|_{C^1(M_\infty,T^*M_\infty\otimes T^*M_\infty)}$; then $Q_k\to 0$. Denote the exterior unit normal and tangent vectors on $\Gamma_\infty$ corresponding to the metrics $\tilde{h}_k$ and $h_\infty$ by $\tilde{\nu}_k$, $\tilde{\gamma}_k$ and $\nu_\infty$, $\gamma_\infty$, respectively.

\smallskip

Denote by $v^f_k$ the harmonic extension of $f\in C^\infty(\Gamma_\infty)$ into $(M_\infty,\beta_k^* h_k)$, i.e., a solution to $\mathfrak{L}_k v^f_k=0$ in $M_\infty$, $v^f_k=f$ on $\Gamma_\infty$. Note that $v^f_k$ and $v^f$ are $C^{r+1}$-smooth due to the Schauder estimates of solutions to elliptic boundary value problems; see Theorem 6.2 and Lemma 6.4, \cite{Trudinger}. Similarly, denote by $v^f$ the harmonic extension of $f$ into the surface $(M_\infty,h_\infty)$. Then $v^f_k-v^f=\tilde{v}_k\in H^1_0(M_\infty)$. Since the multiplication by the $C^1$-smooth function is continuous in $H^{\pm 1}(M_\infty)$, we have
$$\|\mathfrak{L}_k\tilde{v}_k\|_{H^{-1}(M_\infty)}=\|(\mathfrak{L}_k-\mathfrak{L}_\infty)v^f\|_{H^{-1}(M_\infty)}\le c\,Q_k\|v^f\|_{H^1(M_\infty)}\le c(M_\infty,\mathfrak{L}_\infty)Q_k\|f\|_{H^{1/2}(\Gamma_\infty)}.$$
Integration by parts yields
$$\|\nabla_{\tilde{h}_k}\tilde{v}_k\|^2_{L_2(M_\infty;TM_\infty|\,\tilde{h}_k)}=-(\tilde{v}_k,\mathfrak{L}_k\tilde{v}_k)_{L_2(M_\infty;\tilde{h}_k)}\le c(\lambda_+,\mathfrak{L}_\infty)\|\tilde{v}_k\|_{H^{1}(M_\infty)}\|\mathfrak{L}_k\tilde{v}_k\|_{H^{-1}(M_\infty)}.$$
The Friedrichs constant $c_F(\tilde{h}_k)$ of $(M_\infty,\tilde{h}_k)$ obeys
$$c_F(\tilde{h}_k):=\inf_{0\ne v\in H^1_0(M_\infty)}\frac{\int_{M_\infty}\tilde{h}_k^{ij}\partial_i v\partial_j v\sqrt{|\tilde{h}_k|}dx}{\int_{M_\infty}v^2\sqrt{|\tilde{h}_k|}dx}\ge c_F(h_\infty)\lambda_-^2\lambda_+^{-1},$$
whence
$$\|\tilde{v}_k\|_{H^{1}(M_\infty)}\le c(\lambda_\pm,\mathfrak{L}_\infty)\|\nabla_{\tilde{h}_k}\tilde{v}_k\|_{L_2(M_\infty;TM_\infty|\,\tilde{h}_k)}.$$
Combining the last four estimates, one arrives at
\begin{equation}
\label{inner est}
\|\tilde{v}_k\|_{H^{1}(M_\infty)}\le c(\lambda_\pm,\mathfrak{L}_\infty)\|f\|_{H^{1/2}(\Gamma_\infty)}Q_k.
\end{equation}

Now, denote $dl=h_\infty|_{T\Gamma_\infty}$, then $\tilde{h}_k|_{T\Gamma_\infty})=\rho_k dl$, where $\rho_k\to 1$ in $C^1(\Gamma_\infty)$.
Since 
$$\int_{\Gamma_\infty}\partial_{\tilde{\nu}_k}\tilde{v}_k\,v\rho_k dl=(\mathfrak{L}_k\tilde{v}_k,v)_{L_2(M_\infty;\tilde{h}_k)}+(\nabla_{\tilde{h}_k}\tilde{v}_k,\nabla_{\tilde{h}_k}v)_{L_2(M_\infty;TM_\infty|\,\tilde{h}_k)},$$
for any $v\in C^{\infty}(M_\infty)$, we have
\begin{align*}
\Big|\int_{\Gamma_\infty}\partial_{\tilde{\nu}_k}\tilde{v}_k\,v\rho_k dl\Big|\le c(\lambda_\pm,\mathfrak{L}_\infty)(\|\mathfrak{L}_k\tilde{v}_k\|_{H^{-1}(M_\infty)}+\|\tilde{v}_k\|_{H^{1}(M_\infty)}\|v\|_{H^{1}(M_\infty)}\le \\
\le C(\lambda_\pm,\mathfrak{L}_\infty)\|f\|_{H^{1/2}(\Gamma_\infty)}Q_k\|v\|_{H^{1}(M_\infty)},
\end{align*}
due to (\ref{inner est}), whence
$$\|\partial_{\tilde{\nu}_k}\tilde{v}_k\|_{H^{-1/2}(\Gamma_\infty)}\le C(\lambda_\pm,\mathfrak{L}_\infty)\|f\|_{H^{1/2}(\Gamma_\infty)}Q_k
$$
due to the continuity of the multiplication by $C^{1}$-smooth function in $H^{1/2}(\Gamma_\infty)$. Similarly, since
\begin{align*}
\Big|\int_{\Gamma_\infty}(\rho_k\partial_{\tilde{\nu}_k}-\partial_{\nu_\infty})v^f\,vdl\Big|=(\mathfrak{L}_k v^f,v)_{L_2(M_\infty;\tilde{h}_k)}+(\nabla_{\tilde{h}_k}v^f,\nabla_{\tilde{h}_k}v)_{L_2(M_\infty;TM_\infty|\,\tilde{h}_k)}-\\
-(\nabla_{h_\infty}\tilde{v}_k,\nabla_{h_\infty}v)_{L_2(M_\infty;TM_\infty|\,h_\infty)},
\end{align*}
for any $v\in C^{\infty}(M_\infty)$, one has
$$
\|(\partial_{\tilde{\nu}_k}-\partial_{\nu_\infty})v^f\|_{H^{-1/2}(\Gamma_\infty)}\le C(\lambda_\pm,\mathfrak{L}_\infty)\|f\|_{H^{1/2}(\Gamma_\infty)}Q_k
$$
for large $k$. From the last two inequalities, it follows that
$$\|(\tilde{\Lambda_k}-\Lambda_\infty)f\|_{H^{-1/2}(\Gamma_\infty)}=\|\partial_{\tilde{\nu}_k}v_k^f-\partial_{\nu_\infty}v^f\|_{H^{-1/2}(\Gamma_\infty)}\le C(\lambda_\pm,\mathfrak{L}_\infty)\|f\|_{H^{1/2}(\Gamma_\infty)}Q_k,$$
whence
\begin{equation}
\label{closeness of DNs}
\|\tilde{\Lambda}_k-\Lambda_\infty\|_{B(H^{1/2}(\Gamma_\infty);H^{-1/2}(\Gamma_\infty))}\to 0.
\end{equation}

\smallskip

3) Let $\tilde{H}_k:=\partial_{\tilde{\gamma}_k}^{-1}\tilde{\Lambda}_k$ be the Hilbert transforms of $(M_\infty,\beta_k^* h_k)$ and let $\lambda_{\pm l}(\tilde{H}_k)$ be the eigenvalues of $\tilde{H}_k$ different from $\pm i,0$. Since $h_k\in[g_k]$, each diffeomorphism $\beta_k:\,(M_\infty,\beta_k^* h_k)\to (M_k,g_k)$ is conformal and formula (\ref{Hilbert conformal gauge}) yields
\begin{equation}
\label{change of variables}
H_k f=\beta_k^{*-1}\tilde{H}_k\beta_k^*f+{\rm const}(f) \qquad \forall f\in H^{1/2}(\Gamma).
\end{equation}
Note that each $\beta_k^*:\,H^{1/2}(\Gamma)\to H^{1/2}(\Gamma_\infty)$ is continuous due to the $C^r$-smoothness of $\beta_k$. Thus, due to (\ref{change of variables}), the spectra (eigenvalues and their multiplicities) of $H_k$ and $\tilde{H}_k$ (as operators acting in $H^{1/2}(\Gamma)$ and $H^{1/2}(\Gamma_\infty)$) do coincide,
\begin{equation}
\label{specta invariance}
\lambda_{\pm l}(\tilde{H}_k)=\lambda_{\pm l}(H_k) \qquad (l=1,\dots,m).
\end{equation}

\smallskip

4) Introduce the Hilbert transform $H_\infty:=\partial_{\gamma_\infty}^{-1}\Lambda_\infty$ and denote by $\lambda_{\pm l}(H_\infty)$ the eigenvalues of $H_\infty$ different from $\pm i,0$. The convergence $\rho_k\to 1$ in $C^1(\Gamma_\infty)$ implies
$$\|\partial_{\tilde{\gamma}_k}^{-1}-\partial_{\gamma_\infty}^{-1}\|_{B(H^{-1/2}(\Gamma_\infty);H^{1/2}(\Gamma_\infty))}\to 0.$$
Since $$\tilde{H}_k-H_\infty=(\partial_{\tilde{\gamma}_k}^{-1}-\partial_{\gamma_\infty}^{-1})\tilde{\Lambda}_k+(\partial_{\tilde{\gamma}_k}^{-1}-\partial_{\gamma_\infty}^{-1})(\Lambda_\infty-\tilde{\Lambda}_k)+\partial_{\gamma_\infty}^{-1}(\tilde{\Lambda}_k-\Lambda_\infty),$$
the above formula and convergence (\ref{closeness of DNs}) lead to
\begin{equation}
\label{defect closeness}
\|\tilde{H}_k-H_\infty\|_{B(H^{1/2}(\Gamma_\infty))}\to 0.
\end{equation}
Since the discrete spectra of the operators in Banach space are stable under small (in the operator norm) perturbations (see IV.3.5, Kato \cite{Kato}), formula (\ref{defect closeness}) implies that
$$\lambda_{\pm l}(\tilde{H}_k)\to \lambda_{\pm l}(H_\infty) \qquad (l=1,\dots,m).$$
Thus, all the $\lambda_{\pm l}(\tilde{H}_k)$ are not close to $\pm i$ for large $k$. Due to (\ref{specta invariance}), the same is valid for $\lambda_{\pm l}(H_k)$ which contradicts the assumption $\lambda_m(H_k)\to i$. Thus, $Sys(2M_k)\to 0$. Thereby Propositions \ref{real prop} and \ref{main prop} are proved. \end{proof}

\smallskip

\newpage

\paragraph{Connection between the spectrum of the Hilbert transform of a surface and the periods of Abelian differentials on its double.} Chose the {\it symmetric} canonical homology basis $(a_{+1},\dots,a_{2m},b_{+1},\dots,b_{2m})$ on the double $2M$ of $M$ (equipped with the involution $\tau$) in such a way that the loops $\{l_k:=a_{k},l_{m+k}:=b_{k}\}_{k=1}^m$ belong to $M\subset 2M$ while $\{a_{m+k},b_{m+k}\}_{k=1}^m$ belong to $\tau(M)$, and the following symmetry conditions hold
$$a_{m+k}(t)=\tau\circ a_{\pm k}(t), \quad b_{m+k}(t)=\tau\circ b_{\pm k}(1-t), \qquad (t\in[0,1]).$$
Let $\mathscr{D}$ be the space of harmonic differentials $q$ on $M$ which are normal to $\partial M$, $q(\gamma)=0$ ($\gamma$ is the unit tangent vector on $(\Gamma,dl)$). Chose the basis $\{q_i\}_{i=1}^{2m}$ in $\mathscr{D}$ dual to the homology basis $\{l_k:=a_{k},l_{m+k}:=b_{k}\}_{k=1}^m$, $\int_{l_j}q_i=\delta_{ij}$, and introduce the {\it auxiliary period matrix} $\mathfrak{B}$ with entries 
$$\mathfrak{B}_{ij}:=\int_{l_i}\star q_j$$
(here $\star$ is the Hodge operator). 

\begin{prop}
\label{isospectrlity}
The discrete spectrum of the Hilbert transform $H$ with zero excluded do coincide with the spectrum of the auxiliary period matrix $\mathfrak{B}$.
\end{prop}
\begin{proof}
Suppose that $\mathfrak{B}T=\lambda T$, put $q=\sum_{j}T^j q_j$, and introduce the differential $\omega:=(\star-\lambda)q$. Then $\int_{l_i}\star\omega=[(\mathfrak{B}-\lambda)T]^i=0$, whence $\omega=du$ for some harmonic function $u=u^f$. In particular, one has 
$$\Lambda f=\partial_\nu u^f=\omega(\nu)=(\star q,\nu)-\lambda q(\nu)=-(q,\gamma)-\lambda q(\nu)=-\lambda q(\nu)$$ 
and 
$$\partial_\gamma f=\omega(\gamma)=q(\nu)-\lambda q(\gamma)=q(\nu)$$
due to $\star(\nu^\flat)=\gamma^\flat$. Combining the last two equalities, one arrives to $Hf=\partial_\gamma^{-1}\Lambda f=-\lambda f+{\rm const}$, where the constant can be excluded by the choice of $f$. Note that $f={\rm const}$ if and only if $q(\nu)=0$ i.e. $q\equiv 0$ (the latter is equivalent to $T=0$).

\smallskip

Conversely, suppose that $Hf=-\lambda f$, i.e. $\Lambda f=-\lambda\partial_\gamma f$. Recall the $L_2(M;T^*M)$-orthogonal decomposition
\begin{equation}
\label{ortodecomp}
\mathscr{H}^1=d\mathscr{H}^0\oplus\star\mathscr{D}
\end{equation}
of the space of harmonic 1-forms $\mathscr{H}^1=\{a \ | \ da=d\star a \text{ in } M\}$, where $\mathscr{H}^0=\mathscr{H}^1\cap dH^1(M)$ is the sub-space of potential fields. As a corollary, one has
$$\star du^f=du^h+\star p, \qquad \star du^h=du^r+\star q,$$
where $p,q=\sum_j T^jq_j\in\mathscr{D}$. Comparing the above equalities yields
$\star q-p=du^{r-f}$ while restricting them on $\Gamma$ yields
$$\Lambda r=[\star(du^h-q)](\nu)=-\partial_\gamma h, \qquad \Lambda h=-\partial_\gamma f=\lambda^{-1}\Lambda f,$$
whence $h=\lambda^{-1}f+{\rm const}$, $\Lambda r=-\lambda^{-1}\partial_\gamma f=\lambda^{-2}\Lambda f$, i.e. $r=\lambda^{-2}f+{\rm const}$. Therefore $\star q-p=(\lambda^{-2}-1)du^{f}$ and, in particular, one has
$$p(\nu)=-[\star q-p](\nu)=(1-\lambda^{-2})\Lambda f=(\lambda^{-2}-1)\lambda\partial_\gamma f$$
and
$$q(\nu)=[\star q](\gamma)=(\lambda^{-2}-1)\partial_\gamma f.$$
Comparing the last two formulas, one obtains $p(\nu)=\lambda q(\nu)$ on $\Gamma$, whence $p=\lambda q$. As a corollary, $(\star-\lambda)q=(\lambda^{-2}-1)du^{f}$ has zero periods, i.e. $\lambda T^i=\int_{l_i}\star q=\mathfrak{B}_{ij}T^j$.

\smallskip

So, it is shown that $\lambda\ne 0$ is an eigenvalue of $\mathfrak{B}$ if and only if $-\lambda$ is an eigenvalue of $H$ of the same multiplicity. To complete the proof, it remains to note that the spectrum of $H$ is symmetric with respect to the inversion $\lambda\mapsto-\lambda$ since $H\eta=\lambda\eta \ \Rightarrow \ H\overline{\eta}=\overline{\lambda\eta}=-\lambda\overline{\eta}$.
\end{proof}

\smallskip

Now, let $q=\sum_j T^jq_j\in\mathscr{D}$. Since $q$ is normal to $\partial M$, it admits the harmonic extension (which is still denoted by $q$) to the double $2M\supset M$ by the rule $\tau^*q=-q$. Then $(\star-i)q=\omega$ is an Abelian differential on $2M$. The periods of $\omega$ are given by 
$$\int_{l_i}\omega=\mathfrak{B}_{ij}T^j-iT^i, \quad \int_{\tau\circ l_i}\omega=\sum_j T^j\int_{\tau\circ l_i}(\star-i)q_j=\sum_j T^j\int_{l_i}(\star+i)q_j=\mathfrak{B}_{ij}T^j+iT^i$$
(here we used the equality $\tau^*\star=-\star\tau^*$ for the anti-holomorphic involution $\tau$). In particular, if $\mathfrak{B}T=\lambda T$, then
$$\int_{\tau\circ l_i}\omega=\frac{\lambda+i}{\lambda-i}\int_{l_i}\omega \qquad (i=1,\dots,m).$$
So, Proposition \ref{isospectrlity} leads to the following corollary.
\begin{cor}
Let $\lambda_{\pm k}=\pm i\mu_k$ be eigenvalues of the Hilbert transform $H$ of $M$ different from $0,\pm i$. Then in the space of the Abelian differentials on the double $2M$ of $M$ there is the basis $\{\omega_{k},\omega_{-k}\}_{k=1}^m$ with the periods obey
$$\int_{a_{m+i}}\omega_{\pm k}=\frac{\pm \mu_k+1}{\pm\mu_k-1}\int_{a_i}\omega_{\pm k}, \qquad \int_{b_{m+i}}\omega_{\pm k}=\frac{1\pm\mu_k}{1\mp\mu_k}\int_{b_i}\omega_{\pm k} \qquad (i,k=1,\dots,m)$$
in any symmetric homology basis on $2M$.
\end{cor}
Now, suppose that $\lambda_m$ is close to $i$; then on $2M$ there is the Abelian differential $\omega_{-m}$ whose periods on $M$ (thus, its $L_2$-norm) are of order one while the periods on $\tau(M)$ are small (of order $|\lambda_m-i|$). This can be interpreted as a kind of degeneration of $2M$ since any differential $\omega$ on any double $2M$ with all zero periods on $M$ (equivalently, on $\tau(M)$) is zero. Indeed, otherwise $\omega=dw$ on $M$, where $w$ is holomorphic. At the same time, $\omega=\omega_++i\omega_-$ on $2M$, where $\omega_+=(\omega+\overline{\tau^*\omega})/2$, $\omega_-=(\omega-\overline{\tau^*\omega})/2i$ obey $\tau^*\omega=\overline{\omega}$. The latter means that the $\Re \omega_\pm$ are tangent ($\Re \omega_\pm(\nu)=0$) and $\Im \omega_\pm=\star\Re\omega_\pm$ are normal to $\partial M\equiv\Gamma$. Thus, $dw=q+\star p$ on $M$, where $q,p\in\mathscr{D}$. Due to orthogonal decomposition (\ref{ortodecomp}) and the $L_2(M;T^*M)$-orthogonality of the spaces $\mathscr{D}$ and $\star\mathscr{D}$, one has $q=0$. Since $w$ is holomorphic, one has $dw=i\star dw=-ip$, whence $p=0$ and $w=0$. Thus, $\omega=0$ on $2M$.

\smallskip

\paragraph{Acknowledgements.} The author is grateful to Prof. Sergei Ivanov for pointing out the result of \cite{AKKYLT} which significantly simplified the original proof of Proposition \ref{main prop} made in \cite{ArXiv orig}. Also, the author is grateful to Prof. Gaiane Panina for valuable consultations on the moduli spaces of hyperbolic surfaces.

\smallskip

\paragraph{Funding.} The author was supported by the Ministry of Science and Higher Education of the Russian Federation (agreement 075-15-2025-344 dated 29/04/2025 for Saint Petersburg Leonhard Euler International Mathematical Institute at PDMI RAS).

\end{document}